\documentclass{amsart}
\usepackage{graphicx, mathtools}
\usepackage{amsmath,bm}
\usepackage{stmaryrd}
\usepackage{amssymb, comment}
\usepackage[all,cmtip]{xy}
\usepackage{color}
\usepackage{bbold}
\usepackage[hidelinks]{hyperref}
\usepackage{colonequals}
\usepackage{cite}

\newcommand{\Hom}{\operatorname{Hom}\nolimits}

\renewcommand{\mod}{\operatorname{mod}\nolimits}

\newcommand{\depth}{\operatorname{depth}\nolimits}

\newcommand{\Tor}{\operatorname{Tor}\nolimits}
\newcommand{\Ext}{\operatorname{Ext}\nolimits}

\newcommand{\HH}{\operatorname{HH}\nolimits}

\renewcommand{\H}{\operatorname{H}\nolimits}

\newcommand{\cx}{\operatorname{cx}\nolimits}

\newcommand{\lcm}{\operatorname{lcm}\nolimits}

\newcommand{\comments}[1]{}

\newcommand{\Ktac}{\operatorname{\mathsf{K}_{tac}}\nolimits}
\newcommand{\sfT}{\mathsf{T}}
\newcommand{\sfD}{\mathsf{D}}
\newcommand{\sfDb}{\mathsf{D}^{b}}
\newcommand{\sfKb}{\mathsf{K}^{b}}
\newcommand{\slashes}{\operatorname{\!\sslash\!}\nolimits}
\newcommand{\Cent}[1]{\mathsf{Z}(#1)}
\def\id{\operatorname{id}}

\renewcommand{\le}{\leqslant} 
\renewcommand{\ge}{\geqslant}

\newcommand{\fm}{\frak{m}}

\newcommand{\QQ}{\mathbb{Q}}

\newcommand{\ZZ}{\mathbb{Z}}

\newcommand{\noeth}{\mathsf{noeth}}
\newcommand{\noethfl}{\mathsf{noeth}^{\mathsf{fl}}}
\newcommand{\art}{\mathsf{art}}
\newcommand{\artfl}{\mathsf{art}^{\mathsf{fl}}}

\newtheorem{theorem}{Theorem}[section]

\newtheorem{corollary}[theorem]{Corollary}

\newtheorem{lemma}[theorem]{Lemma}
\newtheorem{proposition}[theorem]{Proposition}
\theoremstyle{definition}
\newtheorem{definition}[theorem]{Definition}
\newtheorem{example}[theorem]{Example}
\newtheorem{remark}[theorem]{Remark}

\theoremstyle{definition}

\theoremstyle{definition}

\theoremstyle{definition}
\theoremstyle{definition}

\newtheorem{chunk}[theorem]{}
\theoremstyle{definition}

\theoremstyle{remark}

\theoremstyle{definition}

\thanks{Part of this work was completed at the Centre International de Rencontres Mathématiques (CIRM) in Luminy, Marseille, France during a visit in January 2025. The authors are grateful for their support.}

\begin{document}

\title[Multiplicity in triangulated categories]{Multiplicity in triangulated categories}
\author{Petter Andreas Bergh}
\address[P. A. Bergh]{Institutt for matematiske fag, NTNU, N-7491 Trondheim, Norway}
\email{petter.bergh@ntnu.no}
\urladdr{https://www.ntnu.edu/employees/bergh}
\author{David A. Jorgensen}
\address[D. A. Jorgensen]{Department of Mathematics, University of Texas at Arlington, 411 S. Nedderman Drive, Pickard Hall 429, Arlington, TX 76019, USA}
\email{djorgens@uta.edu}
\urladdr{http://www.uta.edu/faculty/djorgens/}
\author{Peder Thompson}
\address[P. Thompson]{Division of Mathematics and Physics, M{\"a}lardalen University, V{\"a}ster{\aa}s, Sweden}
\email{peder.thompson@mdu.se}
\urladdr{https://sites.google.com/view/pederthompson}

\date{June 3, 2025}                                           

\maketitle

\begin{abstract}
We lay out the theory of a multiplicity in the setting of a triangulated category having a central ring action from a graded-commutative ring $R$, in other words, an $R$-linear triangulated category. The invariant we consider is modelled on those for graded modules over a commutative graded ring.  We show that this invariant is determined by the leading coefficients of the Hilbert polynomials expressing the lengths of certain Hom sets. Our theory is a natural analogue of Hochster's theta invariant for homology and Buchweitz's Herbrand difference for cohomology. Moreover, we give applications to vanishing of cohomology and modules over local complete intersection rings, group algebras of a finite group, and certain finite dimensional algebras. 
\end{abstract}

\section{Introduction}
Multiplicity is a measure of how complicated an algebro-geometric event is.
The origin of multiplicity in commutative algebra stems from classic work of Hilbert \cite{Hil1890} on the growth of graded modules, and from Samuel's development of this idea over local rings \cite{Sam51}. Serre subsequently linked the Hilbert--Samuel multiplicity to intersection theory in algebraic geometry \cite{Ser00}, initiating a flourish of work on related conjectures. We build on these foundations---inspired also by work of Hochster \cite{Hoc81} and Buchweitz \cite{Buc86}---to introduce a multiplicity theory for pairs of objects in an $R$-linear triangulated category, that is, a triangulated category having a central ring action from a graded-commutative ring $R$. The invariant we introduce can be seen as a measure of how complicated cohomology is in this setting. One of our aims, inspired by a fundamental result of Serre \cite{Ser00} and Auslander and Buchsbaum \cite{AB58}, is to show that this multiplicity computes an Euler characteristic of Koszul homology.  We also give several applications of our theory.

Let $R=\oplus_{n\ge 0}R^n$ be a graded-commutative Noetherian ring and $\sfT$ be an $R$-linear triangulated category with shift functor $\Sigma$. There are many rich examples of such categories, see for example  \cite{BIK08} and \cite{BIKO10}. For a pair of objects $X$ and $Y$ in $\sfT$, write $\Hom_{\sfT}(X,Y)$ for the abelian group of morphisms from $X$ to $Y$ in $\sfT$, and consider the graded $R$-module 
$\Hom_{\sfT}^*(X,Y)=\oplus_{n\in \ZZ} \Hom_{\sfT}(X,\Sigma^nY)$. 
Assume that 
\begin{enumerate}
\item The $R$-module $\Hom_{\sfT}^{\ge n_0}(X,Y)$ is Noetherian for some $n_0\in \ZZ$, and 
\item The length of $\Hom_{\sfT}(X,\Sigma^i Y)$ is finite over $R^0$ for all $i\in \ZZ$.
\end{enumerate}
In this case, the lengths $\ell_{R^0}(\Hom_{\sfT}(X,\Sigma^nY))$, for $n\gg0$, are given by polynomials of degree one less than $\cx(X,Y)$, the complexity of $(X,Y)$. 

Assume here that $R$ is generated in degree $d$ for an even positive integer $d$. We define a generalized Herbrand difference---motivated Buchweitz's Herbrand difference \cite{Buc86} and Hochster's theta invariant \cite{Hoc81}---of the pair $(X,Y)$ as a numerical function $h(X,Y):\ZZ\to \ZZ$ by
$$
h(X,Y)(n)\colonequals \sum_{i=0}^{d-1}(-1)^{n+i}\ell_{R^0}\Hom_{\sfT}(X,\Sigma^{n+i}Y)
$$
For $s\ge \cx(X,Y)\ge 1$, the index $s-1$ difference of $h(X,Y)$ yields a multiplicity 
$e^s(X,Y)$; see Definition \ref{e_dfn}.

Our first goal is to prove some basic facts about this multiplicity. For example, for $n\gg0$ the lengths $\ell_{R^0}\Hom_{\sfT}(X,\Sigma^{n+i}Y)$ are expressed by $d$ Hilbert polynomials, and the alternating sum of the coefficients of the degree $s-1$ terms of these polynomials yields 
$e^s(X,Y)$ (Proposition \ref{mult_coeffs}).  We also show our invariant alternates sign with shifts in $\sfT$ (Proposition \ref{shift}), and it is additive on distinguished triangles (Theorem \ref{additivity}). Motivated by the classic axiomatic description of multiplicity in \cite{AB58}, we also provide a similar axiomatic description of this multiplicity in the triangulated setting (Proposition \ref{axioms}).

One of our primary aims, however,  is to show that this multiplicity can be computed as an Euler characteristic of Koszul homology. This is inspired by the foundational result of Serre \cite{Ser00} and Auslander and Buchsbaum \cite{AB58} for Hilbert--Samuel multiplicity, and which has also been explored for other multiplicities, see for example \cite{BCHN23}. We show in Corollary \ref{ABSthm}:
\begin{theorem}
Let $s=\cx(X,Y)\ge 1$ and assume $\Hom_{\sfT}(X,\Sigma^n Y)=0$ for $n\ll0$. If $R^0$ contains an infinite field, or is local with infinite residue field, then there exists a sequence of elements $z_1,...,z_s\in R^d$ such that 
$$ e^s(X,Y)=\sum_{n\in \ZZ}(-1)^n\ell_{R^0}\Hom_{\sfT}(X,\Sigma^n(Y\slashes (z_1,...,z_s)))$$
\end{theorem}
\noindent
Here, for an object $Y$ and an element $z\in R$, the object $Y\slashes z$ is the Koszul object on $z$; see \ref{les} for details. 

One application of this invariant involves vanishing of cohomology. Unlike the classic Hilbert--Samuel multiplicity, where (non)vanishing simply detects dimension, the vanishing of multiplicity here is more subtle. This is expected, in view of recent work on the vanishing of Hochster's theta invariant and the higher Herbrand difference, for example \cite{Wal17,Cel20,BMTW17mf,Wal16}. We show in Theorem \ref{vanishing_app}:
\begin{theorem}
Let $s=\cx(X,Y)$. Assume $e^s(X,Y)=0$.  There is an integer $m_0$ such that if $\Hom_{\sfT}(X,\Sigma^nY)=0$ for $d/2$ consecutive even (or odd) $n\ge m_0$, then $\Hom_{\sfT}(X,\Sigma^nY)=0$ for all $n\ge m_0$. 
\end{theorem}
This leads to a natural open question: Under what conditions does $e^s(X,Y)$ vanish? We explore examples of (non)vanishing of this invariant in Section \ref{section_apps}.

The paper is organized as follows. We begin in Section \ref{section_prelims} with details about the overall setting, as well as some of the key long exact sequences we will use throughout. We give in Section \ref{section_herbrand} the definition of a generalized Herbrand difference and some of its basic properties. In Section \ref{section_mult}, we define the $s$-multiplicity, and prove a number of results about it including the main result discussed here. We explore the homological (or ``negative") side in Section \ref{section_negative}, where it is instead assumed that $\Hom_{\sfT}^{\le n_0}(X,Y)$ is Artinian for some integer $n_0$. In Section \ref{section_apps}, we explore examples and applications, showing that we recover Serre's intersection multiplicity \cite{Ser00}, Hochster's theta invariant \cite{Hoc81}, and Buchweitz's Herbrand difference \cite{Buc86}. Finally, in Section \ref{section_lim} we show that our multiplicity can also be expressed as a limit, similar to other modern multiplicities, and in the Appendix we collect some elementary facts about higher index difference operators.

\section{Setting and preliminaries}\label{section_prelims}
Throughout this paper, let $R=\bigoplus_{n\ge 0}R^n$ be a graded-commutative Noetherian ring, and let $\sfT$ be an $R$-linear triangulated category with shift functor $\Sigma$. 
This means that $\sfT$ admits a central ring action from $R$, that is, there is a graded ring homomorphism $\phi:R\to \Cent{\sfT}$, where $\Cent{\sfT}$ is the graded center of $\sfT$, whose degree $n$ component consists of natural transformations $\eta:\id_{\sfT}\to \Sigma^n{}$ satisfying $\eta\Sigma = (-1)^n\Sigma\eta$; see \cite[Notation 4.1]{BIK08}.

Let $X$ and $Y$ be objects in $\sfT$ and write $\Hom_{\sfT}(X,Y)$ for the abelian group of morphisms from $X$ to $Y$ in $\sfT$. There is a graded abelian group,
\[
\Hom_{\sfT}^*(X,Y)=\bigoplus_{n\in \mathbb Z}\Hom_{\sfT}(X,\Sigma^nY)
\]
For $n_0\in \ZZ$, also set $\Hom_{\sfT}^{\ge n_0}(X,Y)=\bigoplus_{n\ge n_0}\Hom_{\sfT}(X,\Sigma^nY)$ and $\Hom_{\sfT}^{\le n_0}(X,Y)=\bigoplus_{n\le n_0}\Hom_{\sfT}(X,\Sigma^nY)$.

It follows from the definition of $R$-linearity that for every object $X$ of $\sfT$ there is a homomorphism of graded rings $\phi_X:R\to \Hom_{\sfT}^*(X,X)$, such that the induced actions of $R$ on $\Hom_{\sfT}^*(X,Y)$ on either side are compatible. Thus each graded abelian group $\Hom_{\sfT}^*(X,Y)$ is in fact a graded $R$-module.

We recall the notion of Koszul objects and establish the fundamental long exact sequences we will later utilize.

\begin{chunk}\label{les}
Let $Y$ be an object in $\sfT$ and let $z\in R$ be a homogeneous element of degree $d$.  Let $Y\slashes z$ be the object in $\sfT$ defined by the distinguished triangle
\[\xymatrix{
Y\ar[r]^-{\phi_Y(z)} & \Sigma^d{Y} \ar[r] & Y\slashes z \ar[r] & \Sigma{Y}
}\]
This is called the \emph{Koszul object of $z$ on $Y$}. The object $Y\slashes z$ is well-defined (up to a non-unique isomorphism). Given a sequence of elements $\underline{z}=z_1,...,z_r$ in $R$, we inductively define $Y\slashes\underline{z}$ as $(Y\slashes (z_1,...,z_{r-1}))\slashes z_r$; see \cite[Definition 5.10]{BIK08}. 

Let $X$ be another object in $\sfT$. Application of the exact functor $\Hom_{\sfT}(X,-)$ to the previous exact triangle, along with the fact that $\phi(z):\id_{\sfT}\to \Sigma^d$ is a natural transformation, produces the following exact sequence of abelian groups; see also \cite[(5.10.1)]{BIK08}: 
\begin{equation*}\label{les_HomY}
\xymatrix@C=1em{
 \cdots \ar[r] &\Hom_{\sfT}(X,\Sigma^n{Y}) \ar[r]^-{\cdot z} & \Hom_{\sfT}(X,\Sigma^{n+d}{Y})\ar[r] & \Hom_{\sfT}(X,\Sigma^n(Y\slashes z))
\ar@{->} `r[d] `[l] `[dlll] `[l] [dll] & \\ 
&\Hom_{\sfT}(X,\Sigma^{n+1}{Y}) \ar[r]^-{\cdot z} & \Hom_{\sfT}(X,\Sigma^{n+d+1}{Y})\ar[r] & \Hom_{\sfT}(X,\Sigma^{n+1}(Y\slashes z)) \ar[r] &\cdots
}
\end{equation*}
Here, $\Hom_{\sfT}(X,\Sigma^{n}Y)\xrightarrow{\cdot z} \Hom_{\sfT}(X,\Sigma^{n+d}{Y})$ indicates the map $\Hom_{\sfT}(X,\Sigma^n\phi_{Y}(z))$ for each $i\in \mathbb{Z}$. Admittedly, the notation $\cdot z$ is ambiguous, as it could equally indicate the map $\Hom_{\sfT}(X,\phi_{\Sigma^iY}(z))$, but these morphisms agree up to sign and therefore the ambiguity causes no issues for us. 

Dually, application of the exact functor $\Hom_{\sfT}(-,Y)$ to the distinguished triangle $\xymatrix{X \ar[r]^-{\phi_X(z)} & \Sigma^d{X} \ar[r] & X\slashes z \ar[r] & \Sigma{X}}$ produces the exact sequence:
\begin{equation*}\label{les_HomX}
\xymatrix@C=1em{
 \cdots \ar[r] &\Hom_{\sfT}(X\slashes z,\Sigma^{n-1}{Y}) \ar[r]^-{} & \Hom_{\sfT}(X,\Sigma^{n-d-1}{Y})\ar[r]^-{\cdot z} & \Hom_{\sfT}(X,\Sigma^{n-1}{Y})
\ar@{->} `r[d] `[l] `[dlll] `[l] [dll] & \\ 
&\Hom_{\sfT}(X\slashes z,\Sigma^n{Y}) \ar[r] & \Hom_{\sfT}(X,\Sigma^{n-d}{Y})\ar[r]^-{\cdot z} & \Hom_{\sfT}(X,\Sigma^n{Y}) \ar[r] &\cdots
}
\end{equation*}
\end{chunk}

\begin{remark}
As we shall see, the first long exact sequence above dictates the entire theory we lay out below.
One may refer to this as a covariant multiplicity theory.  The second long exact sequence above would govern a contravariant theory of multiplicity, which we do not include in this paper.
\end{remark}

\begin{chunk}\label{R0-homs} 
Let $f:X_1\xrightarrow{} X_2$ and $g:Y_1\xrightarrow{} Y_2$ be morphisms in $\sfT$, and $X$ and $Y$ objects in $\sfT$. Then the induced maps 
\[
g_*:\Hom_\sfT(X,Y_1)\xrightarrow{}\Hom_\sfT(X,Y_2)
\]
and 
\[
f^*:\Hom_\sfT(X_2,Y)\xrightarrow{}\Hom_\sfT(X_1,Y)
\]
are $R^0$-module homomorphisms. This follows directly from definitions.
\end{chunk}

\section{A generalized Herbrand difference}\label{section_herbrand}
Let $X$ and $Y$ be objects in $\sfT$, and set, for $n\in \ZZ$,
\[
\lambda^n(X,Y)\colonequals\ell_{R_0}(\Hom_\sfT(X,\Sigma^nY))
\]
as the length of $\Hom_{\sfT}(X,\Sigma^n Y)$ over $R^0$.

The next definition is motivated by Buchweitz's Herbrand difference \cite{Buc86}.
\begin{definition}\label{Herbrand_dfn} 
Let $X$ and $Y$ be objects in $\sfT$ such that $\lambda^n(X,Y)<\infty$ for all $n\in \ZZ$, and let $d$ be a positive integer. We define the \emph{\emph{(}index $d$\emph{)} generalized Herbrand difference} to be the numerical function $h(X,Y):\ZZ\to \ZZ$ as follows:
\begin{equation*}
h(X,Y)(n)\colonequals \sum_{i=0}^{d-1}(-1)^{n+i}\lambda^{n+i}(X,Y)
\end{equation*}
\end{definition}

\begin{chunk}\label{difference_operators}
For a numerical function $f:\mathbb Z \to \mathbb Z$, define the difference operator $\Delta^1$ (of index $d$) on $f$ by
\begin{align*}
&\Delta^1 f(n)\colonequals f(n+d)-f(n)
\end{align*} 
Complement this by defining $\Delta^0f=f$ and extend it by setting
$\Delta^sf=\Delta^1(\Delta^{s-1}f)$, for $s\ge 1$. 
A closed form expression for $\Delta^sf$, with $s\ge 0$, is given by
\begin{equation}\label{Deltas+}\tag{$\dagger$}
\Delta^sf(n)=\sum_{i=0}^s(-1)^i\binom{s}{i}f(n+(s-i)d)
\end{equation}
The proof of this is delegated to the appendix; see \ref{proofs_of_closed_forms}.
\end{chunk}

\begin{proposition} 
Let $X$ and $Y$ be objects in $\sfT$, $d$ be a positive integer and $z\in R^d$.  If $\lambda^n(X,Y)<\infty$ for all $n\in\ZZ$, then one also has $\lambda^n(X,Y\slashes z)<\infty$ for all $n\in\mathbb Z$.
\end{proposition}

\begin{proof}  
The distinguished triangle
$\xymatrix{Y \ar[r]^-{\phi_Y(z)} &\Sigma^d Y \ar[r] & Y\slashes z \ar[r] & \Sigma Y}$ induces, by \ref{les}, a long exact sequence
\[
\xymatrix@C=1.5em{
\cdots\ar[r] & \Hom_\sfT(X,\Sigma^{n+d}Y)\ar[r] & \Hom_\sfT(X,\Sigma^n(Y\slashes z))\ar[r] &
\Hom_\sfT(X,\Sigma^{n+1}Y) \ar[r]^-{\cdot z} & \cdots
}\]
By \ref{R0-homs}, this is moreover a long exact sequence of $R^0$-modules and $R^0$-module homomorphisms.  Since $\Hom_\sfT(X,\Sigma^n Y)$ has finite length over $R^0$ for all $n\in\mathbb Z$, so does $\Hom_\sfT(X,\Sigma^n (Y\slashes z))$ for all $n\in\mathbb Z$. 
\end{proof}

We end this section with a lemma, needed later to prove Theorem \ref{zee}, regarding application of the difference operator to the generalized Herbrand difference.

\begin{lemma}\label{coinduct+} 
Let $X$ and $Y$ be objects in $\sfT$, $d$ be a positive even integer and $z\in R^d$. Assume that 
$\lambda^n(X,Y)<\infty$ for all $n\in\ZZ$. If for all $n\gg0$ we have short exact sequences
\[
0\to \Hom_\sfT(X,\Sigma^nY) \to \Hom_\sfT(X,\Sigma^{n+d}Y) \to \Hom_\sfT(X,\Sigma^{n}(Y\slashes z)) \to 0
\]
then 
\[
h(X,Y\slashes z)(n)=h(X,Y)(n+d)-h(X,Y)(n)
\]
for all $n\gg 0$, and for $s\ge 1$ we have 
\[
\Delta^{s-1}h(X,Y\slashes z)(n)=\Delta^sh(X,Y)(n)
\]
for all $n\gg 0$.
\end{lemma}

\begin{proof} The short exact sequences give the equations 
\[
\lambda^n(X,Y\slashes z)=\lambda^{n+d}(X,Y)-\lambda^n(X,Y)
\] 
for all $n\gg 0$. Thus for all $n\gg 0$ we have
\begin{align*}
h(X,Y\slashes z)(n)=&\sum_{i=0}^{d-1}(-1)^{n+i}\lambda^{n+i}(X,Y\slashes z)\\
=&\sum_{i=0}^{d-1}(-1)^{n+i}\left(\lambda^{n+d+i}(X,Y)-\lambda^{n+i}(X,Y)\right)\\
=&\sum_{i=0}^{d-1}(-1)^{n+d+i}\lambda^{n+d+i}(X,Y)-
\sum_{i=0}^{d-1}(-1)^{n+i}\lambda^{n+i}(X,Y)\\
=&h(X,Y)(n+d)-h(X,Y)(n)
\end{align*}
which is the first statement.

We prove the second statement by induction on $s\ge 1$. The $s=1$ case is precisely the equality we just established. For $s>1$, we have
\begin{align*}
\Delta^{s-1}h(X,Y\slashes z)(n) &= \Delta^1(\Delta^{s-2}h(X,Y\slashes z))(n)\\
&=\Delta^{s-2}h(X,Y\slashes z)(n+d)-\Delta^{s-2}h(X,Y\slashes z)(n)\\
&=\Delta^{s-1}h(X,Y)(n+d)-\Delta^{s-1}h(X,Y)(n)\\
&=\Delta^1(\Delta^{s-1}h(X,Y))(n)\\
&=\Delta^s h(X,Y)(n)
\end{align*}
where induction occurs at the third equality.
\end{proof}

\section{Covariant Multiplicity}\label{section_mult}
Let $X$ and $Y$ be objects of $\sfT$. As in \cite{AI07}, we say that the graded $R$-module $\Hom_{\sfT}^*(X,Y)$ is \emph{eventually Noetherian} if 
$\Hom_{\sfT}^{\ge n_0}(X,Y):=\bigoplus_{n\ge n_0}\Hom_{\sfT}(X,\Sigma^nY)$ is Noetherian for some $n_0\in \ZZ$. Write $\noeth(R)$ for the category of eventually Noetherian $R$-modules, and $\noethfl(R)$ for the eventually Noetherian $R$-modules of degreewise finite length, that is, where $\lambda^n(X,Y)<\infty$ for all $n\in \ZZ$. Throughout this section, we always assume that $\Hom_{\sfT}^*(X,Y)$ belongs to $\noethfl(R)$.

\begin{remark}\label{examples_rem}
The assumption that $\Hom_{\sfT}^*(X,Y)$ belongs to $\noethfl(R)$ holds in several cases of interest. 

For example, let $(A,\fm)$ be a local complete intersection ring of codimension $c$, and $\sfT=\sfDb(A)$ be the bounded derived category over $A$. This category is $R$-linear, where $R=A[\chi_1,...,\chi_c]$ is the ring of Eisenbud operators over $A$, see for example \cite{AI07}. For complexes $X$ and $Y$ in $\sfDb(A)$, one has that $\Hom^*_{\sfT}(X,Y)$ belongs to $\noeth(R)$; this follows from work of Gulliksen \cite{Gul74}, see also \cite{Avr89,AI07}. Moreover, if $\operatorname{supp}(X)\cap \operatorname{supp}(Y)= \{\fm\}$ then $\Hom_{\sfT}^*(X,Y)$ belongs to $\noethfl(R)$. 

Another example is the stable module category of finitely generated left modules over a group algebra of a finite group over a field, see Evens \cite{Eve61}, Golod \cite{Gol59}, and Venkov \cite{Ven59}. Here the ring acting centrally is the group cohomology ring.

There are also various additional examples for bounded derived categories over a finite dimensional algebra using a central ring action from the Hochschild cohomology ring, see for example Erdmann et.\!\! al.\! \cite{EHSST04}.
\end{remark}

\begin{chunk} \label{dee} Without loss of generality, we assume for the remainder of the paper that $R=R^0[R^d]$ with $d\ge 2$ an even integer. 
Indeed, suppose that $R$ is generated over $R^0$ by $x_1,\dots,x_c$.  Let
$l=\lcm(|x_1|,\dots,|x_c|)$. Then setting 
\[
d=\begin{cases} l  & \text{ if $l$ is even} \\ 2l & \text{ if $l$ is odd} \end{cases}
\]
we see that $d$ is even and $R^0[x_1^{d/|x_1|},\dots,x_c^{d/|x_c|}]$ is Noetherian, over which $R$ is module finite. Any Noetherian $R$-module is also Noetherian over $R^0[x_1^{d/|x_1|},\dots,x_c^{d/|x_c|}]$.
\end{chunk}

\begin{chunk}\label{hilbertpoly}
Since $R$ is generated in degree $d$, the $R$-module $\Hom_{\sfT}^*(X,Y)$ naturally decomposes as a direct sum of $d$ submodules 
\[
\bigoplus_{n\in\mathbb Z}\Hom_{\sfT}(X,\Sigma^{dn+i}Y), \qquad i=0,\dots,d-1
\]
These $d$ submodules are eventually Noetherian, that is, there exists an integer $n_0$ such that $\bigoplus_{n\ge n_0}\Hom_{\sfT}(X,\Sigma^{dn+i}Y)$ is Noetherian.
Then for each $i=0,\dots,d-1$, the Hilbert function 
\[
n\mapsto \lambda^{dn+i}(X,Y) \quad \text{ for }n\ge n_0
\]
is of polynomial type \cite[Corollary 11.2]{AM69}. Consequently, the total Hilbert function $n\mapsto \lambda^n(X,Y)$ for $n\gg0$ is a quasi-polynomial. That is, there exist polynomials $g_0(t),\dots,g_{d-1}(t)$ in $\QQ[t]$, called the Hilbert polynomials of the pair $(X,Y)$, such that 
\[
g_i(n)=\ell_{R_0}(\Hom_\sfT(X,\Sigma^{dn+i}Y))=\lambda^{dn+i}(X,Y) \text{ for all } n\gg0.
\]
\end{chunk}

\begin{definition} The \emph{complexity} of the pair $X$ and $Y$ is defined as
\[
\cx(X,Y)=1+\max\{\deg g_i(t)\mid i=0,\dots,d-1\}
\] 
\end{definition}

By convention we set the degree of the zero polynomial to be $-1$. Thus if each $g_i$ is identically zero, then $\cx(X,Y)=0$.

\begin{remark} Note that the complexity is bounded above by $c$, the number of generators of $R$ over 
$R^0$. When $R_0$ is a field, then if there exists a pair of objects $(X,Y)$ with maximal complexity $c$ then necessarily $R$ is a polynomial ring, i.e., the generators of $R$ over $R^0$ are algebraically independent.
\end{remark}

\begin{chunk}\label{cx_ineq}
Given an object $X$ and a distinguished triangle $Y_1\to Y_2\to Y_3\to \Sigma Y_1$ in $\sfT$, the following inequalities hold for $\{i,j,k\}=\{1,2,3\}$:
\begin{align*}
&\cx(X,Y_i)\le \max\{\cx(X,Y_j),\cx(X,Y_k)\}\\
&\cx(Y_i,X)\le \max\{\cx(Y_j,X),\cx(Y_k,X)\}
\end{align*}
These follow from looking at the long exact sequences
\begin{equation*}\label{les_HomYs}
\xymatrix@C=1em{
 \cdots \ar[r] &\Hom_{\sfT}(X,\Sigma^nY_1) \ar[r] & \Hom_{\sfT}(X,\Sigma^nY_2)\ar[r] & \Hom_{\sfT}(X,\Sigma^nY_3)
\ar@{->} `r[d] `[l] `[dlll] `[l] [dll] & \\ 
&\Hom_{\sfT}(X,\Sigma^{n+1}Y_3) \ar[r]^-{\cdot z} & \Hom_{\sfT}(X,\Sigma^{n+1}Y_2)\ar[r] & \Hom_{\sfT}(X,\Sigma^{n+1}Y_3) \ar[r] &\cdots
}
\end{equation*}
\end{chunk}

For $n\gg0$, the Hilbert functions $n\mapsto \lambda^{dn+i}(X,Y)$ are polynomials of degree 
at most $\cx(X,Y)-1$ for $i=0,\dots,d-1$. Thus for $s\ge \cx(X,Y)$ one has that $\Delta^{s-1}h(X,Y)(n)$ is constant for all $n\gg 0$ (see Theorem \ref{diffpower} in the appendix).

\begin{definition}\label{e_dfn}
Let $X$ and $Y$ be objects in $\sfT$ with $\Hom_{\sfT}^*(X,Y)$ in $\noethfl(R)$.
For an integer $s\ge \cx(X,Y)$, define the \emph{$s$-multiplicity}, $e^s(X,Y)$, of the pair $(X,Y)$ as:
\begin{enumerate}
\item If $\cx(X,Y)\ge 1$, define
\[
e^s(X,Y)\colonequals \Delta^{s-1}h(X,Y)(n) \text{ for all } n\gg 0
\]
\item If $\cx(X,Y)=0$ and $\lambda^{n}(X,Y)=0$ for $n\ll0$, define
\[
e^0(X,Y)\colonequals \sum_{n\in \ZZ}(-1)^{n}\lambda^{n}(X,Y)
\]
\end{enumerate}
\end{definition}

\begin{proposition}\label{mult_coeffs}
Let $X$ and $Y$ be objects in $\sfT$ with $\Hom_{\sfT}^*(X,Y)$ in $\noethfl(R)$, and let $s$ be an integer.
\begin{enumerate}
\item\label{ecx} If $s=\cx(X,Y)\ge 1$, then 
\[
e^s(X,Y)=(s-1)!d^{s-1}\sum_{i=0}^{d-1}(-1)^ia_i
\]
where $a_i\in \QQ$ is the coefficient of the degree $s-1$ term of the Hilbert polynomial $g_i(t)$.
\item If $s>\cx(X,Y)$, then $e^s(X,Y)=0$.
\end{enumerate}
\end{proposition}

\begin{proof} Suppose that $s=\cx(X,Y)\ge 1$.  Then for $n\gg0$ we obtain
\begin{align*}
e^s(X,Y)&=\Delta^{s-1}h(X,Y)(nd)=\Delta^{s-1}\left(\sum_{i=0}^{d-1}(-1)^{dn+i}\lambda^{dn+i}(X,Y)\right)\\
&=\Delta^{s-1}\left(\sum_{i=0}^{d-1}(-1)^{i}g_{i}(n)\right)=\sum_{i=0}^{d-1}(-1)^{i}\Delta^{s-1}g_{i}(n)=\sum_{i=0}^{d-1}(-1)^{i}a_{i}(s-1)!d^{s-1}
\end{align*}
Here we have used the fact that for a polynomial $g(t)=at^{s-1}+\text{(lower degree terms)}$,
$\Delta^{s-1}g(t)=a(s-1)!d^{s-1}$; see Theorem \ref{diffpower} from the appendix. Part (2) follows from (1) along with this fact.
\end{proof}

We next observe how this multiplicity interacts with shifts in $\sfT$.

\begin{proposition}\label{shift}
Let $X$ and $Y$ be objects in $\sfT$ with $\Hom_{\sfT}^*(X,Y)$ in $\noethfl(R)$ and $s$ be an integer with $s\ge \cx(X,Y)\ge 1$. Then 
$$e^s(X,Y)=-e^s(X,\Sigma Y)=-e^s(\Sigma X,Y)$$
These equalities also hold if $\cx(X,Y)=0$ and $\lambda^n(X,Y)=0$ for $n\ll0$.
\end{proposition}
\begin{proof}
The equalities hold trivially by Proposition \ref{mult_coeffs} if $s>\cx(X,Y)$. We justify the first equality, and the second is done similarly.
A key observation used in the following is that $\lambda^n(X,\Sigma Y)=\lambda^{n+1}(X,Y)$ for any $n\in \ZZ$.  
Suppose first that $s=\cx(X,Y)$. 
For $s\ge 1$, by Definitions \ref{e_dfn} and \ref{Herbrand_dfn}, and \eqref{Deltas+}, for all $n\gg0$ one has
\begin{align*}
e^s(X,\Sigma Y)&=\Delta^{s-1}h(X,\Sigma Y)(n)\\
&=\sum_{i=0}^{s-1}(-1)^i{s-1 \choose i}h(X,\Sigma Y)(n+(s-1-i)d)\\
&=\sum_{i=0}^{s-1}(-1)^i{s-1 \choose i} \sum_{j=0}^{d-1}(-1)^{n+(s-1-i)d+j}\lambda_R^{n+(s-1-i)d+j}(X,\Sigma Y)
\end{align*}
Using that $d$ is even, this becomes:
\begin{align*}
e^s(X,\Sigma Y)&=\sum_{i=0}^{s-1}(-1)^i{s-1 \choose i} \sum_{j=0}^{d-1}(-1)^{n+j}\lambda_R^{n+j+1}(X, Y)\\
&=-\sum_{i=0}^{s-1}(-1)^i{s-1 \choose i} \sum_{j=0}^{d-1}(-1)^{n+j+1}\lambda_R^{n+j+1}(X, Y)\\
&=-\sum_{i=0}^{s-1}(-1)^i{s-1 \choose i} h(X, Y)(n+1)\\
&=-\Delta^{s-1}h(X, Y)(n+1)\\
&=-e^s(X, Y)
\end{align*}

Now suppose that $s=0$ and $\lambda^n(X,Y)=0$ for $n\ll0$. In this case, one has
\begin{align*}
e^0(X,\Sigma Y)&=\sum_{n\in \ZZ}(-1)^n\lambda^n(X,\Sigma Y)=\sum_{n\in \ZZ}(-1)^n\lambda^{n+1}(X,Y)\\
&=-\sum_{n\in \ZZ}(-1)^{n+1}\lambda^{n+1}(X,Y)=-e^0(X,Y)
\end{align*}
as desired.
\end{proof}

For a graded $R$-module $M$ and an element $z\in R^d$, we say that $z$ is \emph{eventually injective} on $M$ if the map $M^n\xrightarrow{\cdot z} M^{n+d}$ is injective for $n\gg0$.
\begin{theorem}\label{zee}
Let $X$ and $Y$ be objects in $\sfT$ with $\Hom_{\sfT}^*(X,Y)$ in $\noethfl(R)$.
Assume that $R^0$ either contains an infinite field or is local with infinite residue field. Then there exists an element $z\in R^d$ which is eventually injective on $\Hom_\sfT^{*}(X,Y)$, and for $\cx(X,Y)\ge 1$, one has
\[
\cx(X,Y\slashes z)=\cx(X,Y)-1
\]
If $s\ge\cx(X,Y)\ge 2$, then
\[
e^{s-1}(X,Y\slashes z)=e^s(X,Y)
\]
If $\cx(X,Y)=1$ and $\lambda^i(X,Y)=0$ for $i\ll0$, then
\[
e^{0}(X,Y\slashes z)=e^1(X,Y)
\]
\end{theorem}

\begin{proof} 
First, \cite[Lemma 2.2]{BJT24} yields the existence of an element $z\in R^d$ which is eventually injective on $\Hom_{\sfT}^{*}(X,Y)$. Consequently, for $n\gg0$ there are exact sequences 
\[\xymatrix@C=1em{
0\ar[r] & \Hom_\sfT(X,\Sigma^nY) \ar[r]^-{\cdot z} & \Hom_\sfT(X,\Sigma^{n+d}Y) \ar[r] & \Hom_\sfT(X,\Sigma^{n}(Y\slashes z)) \ar[r] & 0
}\]
This yields, for $n\gg0$, that
\[
\lambda^{dn+i}(X,Y\slashes z)=\lambda^{dn+d+i}(X,Y)-\lambda^{dn+i}(X,Y)
\] 
Let $g_0(t),...,g_{d-1}(t)$ be the Hilbert polynomials belonging the pair $(X,Y)$ as in \ref{hilbertpoly}. For each $i$ we have
$$\lambda^{dn+i}(X,Y\slashes z) = \lambda^{dn+d+i}(X,Y)-\lambda^{dn+i}(X,Y)=g_i(n+1)-g_i(n).$$
If $g_i(t)$ is nonzero, then the Hilbert polynomial of the pair $(X,Y\slashes z)$ is $g_i(t+1)-g_i(t)$, which has degree one less than the degree of $g_i(t)$. Thus the statement about complexity follows.

Now take $s\ge \cx(X,Y)\ge 2$. From Lemma \ref{coinduct+},
\begin{align*}
e^{s-1}(X,Y\slashes z)&=\Delta^{s-2}h(X,Y\slashes z)(n)\text{ for $n\gg0$}\\
&=\Delta^{s-1}h(X,Y)(n)\text{ for $n\gg0$}\\
&=e^s(X,Y).
\end{align*}

Finally, suppose that $\cx(X,Y)=1$ and $\lambda^n(X,Y)=0$ for $n\ll0$. In this case, the long exact sequence from \ref{les} yields that also $\lambda^n(X,Y\slashes z)=0$ for $n\ll0$. By replacing $Y$ by $\Sigma^{2n}(Y)$ for some $n\in \ZZ$ if needed---which leaves $e^{0}(X,Y\slashes z)$ and $e^{1}(X,Y)$ unchanged per Proposition \ref{shift}---we can assume that $\lambda^n(X,Y)=0=\lambda^n(X,Y\slashes z)$ for $n<0$. By the work above, $\cx(X,Y\slashes z)=0$, and so take any integer $n_0>d$ so that both $\lambda^n(X,Y\slashes z)=0$ for $n\ge n_0$ and $e^1(X,Y)=h(X,Y)(n_0+1)$.  Now, counting lengths of the modules in the same exact sequence \ref{les} yields the following:
\begin{align*}
\sum_{n=0}^{n_0} (-1)^n\lambda^{n+d}(X,Y)&=\sum_{n=0}^{n_0} (-1)^n\lambda^{n}(X,Y)+\sum_{n=0}^{n_0}(-1)^n\lambda^n(X,Y\slashes z)
\end{align*}
We thus have the following, noting also that since $\lambda^n(X,Y\slashes z)=0$ for $n<0$, the exact sequence \ref{les} yields that $\lambda^n(X,Y)=\lambda^{n+d}(X,Y)$ for $n<0$, which explains the fifth equality:
\begin{align*}
e^0(X,Y\slashes z)&=\sum_{n=0}^{n_0}(-1)^n\lambda^n(X,Y\slashes z)\\
&=\sum_{n=0}^{n_0} (-1)^n\lambda^{n+d}(X,Y)-\sum_{n=0}^{n_0} (-1)^n\lambda^{n}(X,Y)\\
&=\sum_{n=d}^{n_0+d} (-1)^n\lambda^{n}(X,Y)-\sum_{n=d}^{n_0}(-1)^n\lambda^n(X,Y)-\sum_{n=0}^{d-1}(-1)^n\lambda^n(X,Y)\\
&=\sum_{n=n_0+1}^{n_0+d}(-1)^n\lambda^n(X,Y)-\sum_{n=0}^{d-1}(-1)^n\lambda^n(X,Y)\\
&=\sum_{n=n_0+1}^{n_0+d}(-1)^n\lambda^n(X,Y)-\sum_{n=-d}^{-1}(-1)^n\lambda^n(X,Y)\\
&=\sum_{n=0}^{d-1}(-1)^{n+n_0+1}\lambda^{n+n_0+1}(X,Y)\\
&=h(X,Y)(n_0+1)\\
&=e^1(X,Y)
\end{align*}
This completes the argument.
\end{proof}

The next result gives a relationship between the Euler characteristic of a Koszul object and the multiplicity, analogous to Serre's \cite[IV.A.3 Theorem 1]{Ser00}, which was generalized by Auslander and Buchsbaum's \cite[Proposition 3.5]{AB58} (see also \cite{AB59}).
\begin{corollary}\label{ABSthm}
Let $X$ and $Y$ be objects in $\sfT$ with $\Hom_{\sfT}^*(X,Y)$ in $\noethfl(R)$.
Assume that $R^0$ either contains an infinite field or is local with infinite residue field. If $s=\cx(X,Y)\ge 0$ and $\lambda^n(X,Y)=0$ for $n\ll0$, then there exists a sequence of elements $z_1,...,z_s\in R^d$ such that 
$$ e^s(X,Y)=\sum_{n\in \ZZ}(-1)^n\lambda^n(X,Y\slashes (z_1,...,z_s))$$
\end{corollary}
\begin{proof}
If $s=0$, then the result holds by the definition of $e^0(X,Y)$. Now suppose that $s\ge 1$. Inductively apply Theorem \ref{zee} to obtain a sequence of elements $z_1,...,z_s\in R^d$, so that $z_i$ is eventually injective on $\Hom_{\sfT}^{*}(X,Y\slashes(z_1,...,z_{i-1}))$. By the same theorem, we see that 
$$e^s(X,Y)=e^{s-1}(X,Y\slashes z_1)=\cdots=e^{0}(X,Y\slashes(z_1,...,z_s)).$$  
The result now holds by the definition of $e^0(X,Y)$.
\end{proof}

The next result shows that multiplicity is additive on distinguished triangles.
\begin{theorem}\label{additivity}
Let $X$ and $Y$ be objects in $\sfT$ with $\Hom_{\sfT}^*(X,Y)$ in $\noethfl(R)$.
Assume that $R^0$ either contains an infinite field or is local with infinite residue field.  If
$Y_1 \to Y_2 \to Y_3 \to \Sigma Y_1$ is a distinguished triangle with $\cx(X,Y_i)\le s$ for $i=1,2,3$, then 
\[
e^s(X,Y_2)=e^s(X,Y_1)+e^s(X,Y_3)
\]
holds if $s\ge 1$, or if $s=0$ and $\lambda^i(X,Y_j)=0$ for $j=1,2,3$ and $i\ll0$.
\end{theorem}

\begin{proof} 
We induct on $r=\max\{\cx(X,Y_i)\mid i=1,2,3\}$.  If $s>r$, then the result holds trivially by Proposition \ref{mult_coeffs}, so we assume $s=r$. If $r=0$, then the result follows from the long exact sequence in \ref{les_HomYs} and the definition of $e^0(X,Y_i)$. So we may assume 
$r>0$. If $\cx(X,Y_i)=0$ for some $i$ (but not all), then the conclusion is easily seen to hold. Thus assume $\cx(X,Y_i)>0$ for $i=1,2,3$. For the base case we consider $r=1$.  By Theorem \ref{zee}, we can find $z\in R^d$ such that the induced maps $\Hom_\sfT(X,\Sigma^nY_i)\xrightarrow {\cdot z} \Hom_\sfT(X,\Sigma^{n+d}Y_i)$ are isomorphisms for $i=1,2,3$ and all $n\gg 0$.

Consider the morphism of distinguished triangles
\[\xymatrix{
Y_1\ar[d]_-{\phi_{Y_1}(z)} \ar[r] & Y_2 \ar[r]\ar[d]_-{\phi_{Y_2}(z)}  & Y_3 \ar[d]_-{\phi_{Y_3}(z)} \ar[r] & \Sigma Y_1\ar[d]_-{\Sigma\phi_{Y_1}(z)} \\
\Sigma^dY_1\ar[r]^-{} & \Sigma^dY_2 \ar[r] & \Sigma^dY_3 \ar[r] & \Sigma^{d+1} Y_1
}\]
using the fact that $\phi_{\Sigma Y_1}(z)=\Sigma\phi_{Y_1}(z)$ since $d$ is even. From this diagram, we obtain the following commutative diagram for all $n\gg 0$.
\[\xymatrix@C=5mm{
0 \ar[r] & K \ar[r] &\Hom_\sfT(X,\Sigma^nY_1)\ar[d]_-{\cdot z}^{\cong} \ar[r] & \cdots \ar[r] & \Hom_\sfT(X,\Sigma^{n+d-1}Y_3) \ar[d]_-{\cdot z}^{\cong} \ar[r] & C \ar[r] & 0\\
0 \ar[r] & K' \ar[r] &\Hom_\sfT(X,\Sigma^{n+d}Y_1)\ar[r] & \cdots \ar[r] & \Hom_\sfT(X,\Sigma^{n+2d-1}Y_3) \ar[r] & C' \ar[r] & 0
}\]
Since the vertical maps are isomorphisms, $K\cong K'$ and $C\cong C'$ and by exactness we have $C\cong K'$.
Thus $K\cong C$, and the result follows now by additivity of lengths on the top row.

Now suppose that $r>1$. Note from \ref{cx_ineq} that at least two of $\cx(X,Y_i)$ for $i=1,2,3$ must be $r$. If one of the pairs has complexity $0$, as noted above the result holds easily. Thus assume that $\cx(X,Y_i)>0$ for $i=1,2,3$. Apply Theorem \ref{zee} to the pair $(X,Y_1\oplus Y_2\oplus Y_3)$ to find $z\in R^d$ such that $\cx(X,Y_i\slashes z)=\cx(X,Y_i)-1$ for $i=1,2,3$. By induction, the equality $e^{r-1}(X,Y_2\slashes z)=e^{r-1}(X,Y_1\slashes z)+e^{r-1}(X,Y_3\slashes z)$ holds. If one of these pairs has complexity less than $r$, say $\cx(X,Y_j)<r$, then $e^r(X,Y_j)=0=e^{r-1}(X,Y_j\slashes z)$. Otherwise, $e^r(X,Y_i)=e^{r-1}(X,Y_i\slashes z)$ for $i=1,2,3$ and the result holds.
\end{proof}

The next proposition provides justification that the definition of multiplicity proposed in this paper is the natural definition. That is, it axiomatizes the notion of a multiplicity function in this triangulated setting, analogous to the axioms discussed in \cite[p. 641]{AB58} for the Hilbert--Samuel multiplicity in the local ring setting.

\begin{proposition}\label{axioms}
Let $X$ and $Y$ be objects in $\sfT$ with $\Hom_{\sfT}^*(X,Y)$ in $\noethfl(R)$.
For $s\ge 0$ let $f^s$ be a $\ZZ$-valued function on
$$\{(X,Y)\in \sfT\times \sfT \mid \lambda^n(X,Y)=0\text{ for }n\ll0\text{ and } \cx(X,Y)\le s\}$$
satisfying the following axioms:
\begin{enumerate}
\item If $s>\cx(X,Y)$, then $f^s(X,Y)=0$.
\item If $\cx(X,Y)=0$, then $f^0(X,Y)=\sum_{n\in \ZZ}(-1)^n\lambda^n(X,Y)$.
\item If $s=\cx(X,Y)>0$ and $z\in R^d$ is eventually injective on $\Hom_{\sfT}^*(X,Y)$, then $f^s(X,Y)=f^{s-1}(X,Y\slashes z)$.
\end{enumerate}
If $R^0$ either contains an infinite field or is local with infinite residue field, then $f^s(X,Y)=e^s(X,Y)$.
\end{proposition}
\begin{proof}
If $s>\cx(X,Y)$, then $f^s(X,Y)=0=e^s(X,Y)$ by (1) and Proposition \ref{mult_coeffs}. Thus assume $s=\cx(X,Y)$ and induct on $s$. One has $f^0(X,Y)=e^0(X,Y)$ by (2) and the definition of $e^s(X,Y)$. For $s>0$, we have by (3) and Theorem \ref{zee} that
$$e^s(X,Y)=e^{s-1}(X,Y\slashes z)=f^{s-1}(X,Y\slashes z)=f^s(X,Y)$$
and so $f^s(X,Y)=e^s(X,Y)$.
\end{proof}

\section{The negative side of things}\label{section_negative}
Let $X$ and $Y$ be objects in $\sfT$.  Dual to the previous section, we turn to investigating 
$\lambda^n(X,Y)$ for $n\ll 0$ rather than for $n\gg0$. This is needed for some of the homological applications in the next section. We say that the graded $R$-module $\Hom_{\sfT}^*(X,Y)$ is \emph{eventually Artinian} if $\Hom_{\sfT}^{\le n_0}(X,Y):=\bigoplus_{n\le n_0}\Hom_\sfT(X,\Sigma^nY)$ is Artinian for some $n_0\in \ZZ$. Write $\art(R)$ for the category of eventually Artinian $R$-modules, and $\artfl(R)$ for the eventually Artinian $R$-modules of degreewise finite length, that is, where $\lambda^i(X,Y)<\infty$ for all $i\in \ZZ$. 
Throughout this section, we always assume that $\Hom_{\sfT}^*(X,Y)$ belongs to $\artfl(R)$.

\begin{remark}\label{examples_rem2}
This condition holds in multiple natural settings. For example, the $R$-module $\Hom_{\sfT}^*(X,Y)$ belongs to $\artfl(R)$ for all objects $X$ and $Y$ in the stable module category $\sfT$ over a zero dimensional local complete intersection ring, or over a group algebra of a finite group over a field, where $R$ is as in Remark \ref{examples_rem}.
\end{remark}

Recall that $R$ is generated in an even degree $d$, see \ref{dee}. For a numerical function $f:\mathbb Z \to \mathbb Z$, define the negative difference operator $\Delta^{-1}$ (of index $d$) on $f$ by
\begin{align*}
&\Delta^{-1} f(n)=f(n+1)-f(n+d+1)
\end{align*} 
Note that $\Delta^{-1} f(n)=-\Delta^1f(n+1)$. Complement this by defining $\Delta^0f=f$ and extend it by setting
$\Delta^{-s}f=\Delta^{-1}(\Delta^{-s+1}f)$ for $s\ge 1$. Induction on $s\ge 1$ gives $\Delta^{-s}f(n)=(-1)^s\Delta^sf(n+s)$, and an elementary argument using \eqref{Deltas+} for $\Delta^sf$ then yields a closed form for $\Delta^{-s}f$, for $s\ge 0$:
\begin{equation}\label{Deltas-}\tag{$\dagger\dagger$}
\Delta^{-s}f(n)=\sum_{i=0}^s(-1)^i\binom{s}{i}f(n+di+s)
\end{equation}

\begin{lemma}\label{coinduct-} 
Let $X$ and $Y$ be objects in $\sfT$ with $\Hom_{\sfT}^*(X,Y)$ in $\artfl(R)$, $d$ as in \ref{dee} and $z\in R^d$. If for all $n\ll 0$ we have short exact sequences
\[
0\to \Hom_\sfT(X,\Sigma^{n}(Y\slashes z)) \to \Hom_\sfT(X,\Sigma^{n+1}Y) \to 
\Hom_\sfT(X,\Sigma^{n+d+1}Y) \to  0
\]
then 
\[
h(X,Y\slashes z)(n)=h(X,Y)(n+d+1)-h(X,Y)(n+1)
\]
for all $n\ll 0$, and for $s\ge 1$ we have 
\[
\Delta^{-s+1}h(X,Y\slashes z)(n)=-\Delta^{-s}h(X,Y)(n)
\]
for all $n\ll 0$.
\end{lemma}

\begin{proof} The short exact sequences give the equations 
\[
\lambda^n(X,Y\slashes z)=\lambda^{n+1}(X,Y)-\lambda^{n+d+1}(X,Y)
\] 
for all $n\ll 0$. Thus for all $n\ll 0$ we have
\begin{align*}
h(X,Y\slashes z)(n)=&\sum_{i=0}^{d-1}(-1)^{n+i}\lambda^{n+i}(X,Y\slashes z)\\
=&\sum_{i=0}^{d-1}(-1)^{n+i}\left(\lambda^{n+1+i}(X,Y)-\lambda^{n+d+1+i}(X,Y)\right)\\
=&\sum_{i=0}^{d-1}-(-1)^{n+1+i}\lambda^{n+1+i}(X,Y)+
\sum_{i=0}^{d-1}(-1)^{n+d+1+i}\lambda^{n+d+1+i}(X,Y)\\
=&h(X,Y)(n+d+1)-h(X,Y)(n+1)
\end{align*}
which is the first statement.

As in the proof of Lemma \ref{coinduct+}, the last statement follows by induction on $s$.
\end{proof}

In the case where $\Hom_{\sfT}^*(X,Y)$ belongs to $\artfl(R)$, analogous to \ref{hilbertpoly} we obtain (negative) Hilbert polynomials. In detail: the $d$ submodules 
\[
\bigoplus_{n\in\mathbb Z}\Hom_{\sfT}(X,\Sigma^{dn+i}Y), \qquad i=0,\dots,d-1
\]
are eventually Artinian, meaning that there exists an integer $n_0$ such that each $\bigoplus_{n\le n_0}\Hom_{\sfT}(X,\Sigma^{dn+i}Y)$ is an Artinian $R$-module. The Hilbert function 
\[
n\mapsto \lambda^{dn+i}(X,Y)\quad\text{for $n\le n_0$}
\]
is of polynomial type \cite[Theorem 2(ii)]{Kir73}. Consequently, the total Hilbert function $n\mapsto \lambda^n(X,Y)$ for $n\ll0$ is a quasi-polynomial. That is, there exist polynomials $g_0^-(t),...,g_{d-1}^-(t)$ in $\QQ[t]$, called the negative Hilbert polynomials of the pair $(X,Y)$, such that
$$g_i^-(n)=\lambda^{dn+i}(X,Y)\quad\text{for all $n\ll0$}.$$

\begin{definition}
The \emph{negative complexity} of the pair $X$ and $Y$ is defined as
\[
\cx^-(X,Y)=1+\max\{\deg g^-_i(t)\mid i=0,\dots,d-1\}
\] 
\end{definition}

\begin{definition} 
Let $X$ and $Y$ be objects in $\sfT$ with $\Hom_{\sfT}^*(X,Y)$ in $\artfl(R)$. For an integer $s\ge \cx^-(X,Y)$, define the \emph{negative $s$-multiplicity}, $e_s(X,Y)$, of $(X,Y)$ as:
\begin{enumerate}
\item If $\cx^-(X,Y)\ge 1$, define
\[
e_s(X,Y)\colonequals \Delta^{-s+1}h(X,Y)(n) \text{ for all } n\ll 0
\]
\item If $\cx^-(X,Y)=0$ and $\lambda^n(X,Y)=0$ for $n\gg0$, define
\[
e_0(X,Y)\colonequals \sum_{n\in \ZZ}(-1)^n\lambda^n(X,Y)
\]
\end{enumerate}
\end{definition}
Note that in the case where $e_0(X,Y)$ and $e^0(X,Y)$ are both defined, then $e_0(X,Y)=e^0(X,Y)$.

\begin{proposition}\label{mult_coeffs-}
Let $X$ and $Y$ be objects in $\sfT$ with $\Hom_{\sfT}^*(X,Y)$ in $\artfl(R)$, and $s$ be an integer.
\begin{enumerate}
\item If $s=\cx^-(X,Y)\ge 1$, then 
\[
e_s(X,Y)=(s-1)!d^{s-1}\sum_{i=0}^{d-1}(-1)^ia_i
\]
where $a_i\in \QQ$ is the coefficient of the degree $s-1$ term of $g_i^-(t)$.
\item If $s>\cx^-(X,Y)$, then $e_s(X,Y)=0$.
\end{enumerate}
\end{proposition}
\begin{proof}
The argument is similar to that of Proposition \ref{mult_coeffs}.
\end{proof}

For a graded $R$-module $M$ and an element $z\in R^d$, we say that $z$ is \emph{eventually surjective} on $M$ if the map $M^n\xrightarrow{\cdot z} M^{n+d}$ is surjective for $n\ll0$.

\begin{theorem}\label{zee-} 
Let $X$ and $Y$ be objects in $\sfT$ with $\Hom_{\sfT}^*(X,Y)$ in $\artfl(R)$.
Assume that $R^0$ is local with infinite residue field. Then there exists an element $z\in R^d$ which is eventually surjective on 
$\Hom_\sfT^{*}(X,Y)$, and for $\cx^-(X,Y)\ge 1$ one has
\[
\cx^-(X,Y\slashes z)=\cx^-(X,Y)-1
\]
If $s\ge\cx^-(X,Y)\ge 2$, then
\[
e_{s-1}(X,Y\slashes z)=e_s(X,Y)
\]
If $\cx^-(X,Y)=1$ and $\lambda^i(X,Y)=0$ for $i\gg0$, then
\[
e_{0}(X,Y\slashes z)=e_1(X,Y)
\]
\end{theorem}

\begin{proof}
Let $D(-)=\Hom_{R^0}(-,E(k))$, where $k$ is the residue field of $R^0$ and $E(k)$ its injective envelope. Then $\Hom_{\sfT}^*(X,Y)$ is in $\artfl(R)$ if and only if $D(\Hom_{\sfT}^*(X,Y))$ is in $\noethfl(R)$; this follows from work of Kirby \cite{Kir73}, see also \cite[Theorem 1.4]{BJT25}. Now as in the proof of Theorem \ref{zee} above, \cite[Lemma 2.2]{BJT24} yields the existence of an element $z\in R^d$ which is eventually injective on $D(\Hom_{\sfT}^*(X,Y))$, and is therefore eventually surjective on $\Hom_{\sfT}^*(X,Y)$. The remainder of the proof now follows in a similar fashion as that of Theorem \ref{zee}.
\end{proof}

The local assumption in Theorem \ref{zee-} is not too restrictive; for example, see Remark \ref{examples_rem2}.
To complement Corollary \ref{ABSthm}, we have the following:

\begin{corollary}\label{ABSthm-}
Let $X$ and $Y$ be objects in $\sfT$ with $\Hom_{\sfT}^*(X,Y)$ in $\artfl(R)$. 
Assume that $R^0$ is local with infinite residue field. If $s=\cx^-(X,Y)\ge 0$ and $\lambda^n(X,Y)=0$ for $n\gg0$, then there exists a sequence of elements $z_1,...,z_s\in R^d$ such that 
$$ e_s(X,Y)=\sum_{n\in \ZZ}(-1)^n\lambda^n(X,Y\slashes (z_1,...,z_s))$$
\end{corollary}
\begin{proof}
This is dual to the proof of Corollary \ref{ABSthm}.
\end{proof}

\begin{chunk}
Finally, observe that the negative multiplicity alternates sign when shifting, similar to Proposition \ref{shift}, and it is also additive on distinguished triangles, similar to Theorem \ref{additivity}. Moreover, the negative multiplicity can be axiomatized analogously to Proposition \ref{axioms}.
\end{chunk}

\section{Examples and applications}\label{section_apps}
The aim of this section is to present examples of multiplicity, and to show this multiplicity extends notions of multiplicity in the literature, such as Serre's intersection multiplicity \cite{Ser00}, Hochster's theta invariant \cite{Hoc81}, and Buchweitz's Herbrand difference \cite{Buc86}. In addition, we consider an application to vanishing of cohomology.

We start with some concrete examples showing that the multiplicity can easily vanish, even at the complexity, or be equal to any integer.

\begin{example}
Let $k$ be a field and let $A=k[x]/(x^2)$. The triangulated category $\sfT=\sfD(A)$ is $R$-linear, where $R=A[\chi]$ and $\chi$ has degree $d=2$. Let $X=A/(x)$, considered as a complex concentrated in degree $0$. For all $n\ge 1$,
$$\Hom_{\sfD(A)}(X,\Sigma^nX)=\Ext_A^n(A/(x),A/(x))\cong k$$
and so $\cx(X,X)=1$ with $g_0(t)=1$ and $g_1(t)=1$. Thus $e^1(X,X)$, which is the same as Buchweitz's Herbrand difference \cite{Buc86}, is zero.
\end{example}

\begin{example}
Let $k$ be a field, $r\ge 1$ an integer, and $A=k[x,y]/(xy^r)$. The triangulated category $\sfT=\sfD(A)$ is $R$-linear, where $R=A[\chi]$ and $\chi$ has degree $d=2$. Let $X=A/(x)$, considered as a complex concentrated in degree $0$. Then for $n\ge 1$
\[
\Hom_{\sfD(A)}(X,\Sigma^nX)=\Ext_A^n(A/(x),A/(x))\cong\begin{cases}k[y]/(y^r)& \text{if $n$ is even}\\ 0 & \text{if $n$ is odd}\end{cases}
\]
Thus $\cx(X,X)=1$ and $e^1(X,X)=r$. Also $e^1(X,\Sigma X)=-r$ by Proposition \ref{shift}.
\end{example}

\begin{example}
Let $k$ be a field, $c\ge 1$ an integer, and 
\[
A=k[x_1,...,x_c,y_1,...,y_c]/(x_1y_1,...,x_cy_c)
\] 
Then $\sfT=\sfD(A)$ is $R$-linear, where $R=A[\chi_1,...,\chi_c]$ and each $\chi_i$ has degree $d=2$. Let $X=A/(x_1,...,x_c)$ and $Y=A/(y_1,...,y_c)$. This example was considered in 
\cite[Example 6.2]{JST22}, where it was shown that 
\[
\sum_{n=0}^\infty \ell\Ext_A^n(X,Y)t^n=\frac{t^c}{(1-t^2)^c}
\]
Using the binomial series, and the identity ${-a\choose b}=(-1)^b{a+b-1\choose b}$ for non-negative integers $a$ and $b$, see \cite[1.20]{Sta12}, one has
$$t^c(1-t^2)^{-c}=t^c\sum_{n=0}^\infty(-1)^n{-c\choose n}t^{2n}=\sum_{n=0}^\infty{c+n-1 \choose n}t^{2n+c}$$
It follows that for $c$ even, $\lambda^{2n+1}(X,Y)=0$ for all $n\gg0$, thus $g_1(t)=0$, and 
\[
\lambda^{2n}(X,Y)={n+c/2-1\choose n-c/2}=\frac{1}{(c-1)!}(n+c/2-1)(n+c/2-2)\cdots(n-c/2+1)
\] 
for all $n\gg 0$. From this we read off $g_0(t)=\frac{1}{(c-1)!}t^{c-1}+h(t)$ where $h(t)$ is a polynomial of degree $c-2$. Thus $\cx(X,Y)=c$ and by Proposition \ref{mult_coeffs} one has 
$e^c(X,Y)=2^{c-1}$. Similarly, for $c$ odd one has $e^c(X,Y)=-2^{c-1}$.
\end{example}

Now we turn to some examples in other settings.

\begin{example} Consider the symmetric group on 4 letters $S_4$ and $\mathbb F_2$, the field of two elements.  Let $\sfT=\underline{\mod}(\mathbb F_2S_4)$ be the stable module category of finitely generated $\mathbb F_2S_4$-modules. Nakaoka \cite{Nakaoka1962} describes completely the cohomology ring
$H=\bigoplus_{n\ge 0}\underline{\Hom}_{\mathbb F_2S_4}(\mathbb F_2,\Omega^{-n}\mathbb F_2)
=\Ext_{\mathbb F_2S_4}^{\ge 0}(\mathbb F_2,\mathbb F_2)$; it is the commutative ring
$H=\mathbb F_2[x,y,z]/(xz)$ generated over $\mathbb F_2$ by three generators $x,y,z$ with degrees $|x|=1$, $|y|=2$, and $|z|=3$, and subject to the relation $xz=0$. It is easy to see that the Hilbert series of $H$ is given by the rational function $\frac{1-t^4}{(1-t)(1-t^2)(1-t^3)}$. As per \ref{dee}
one may take for $R$ the subalgebra $R=\mathbb F_2[x^6,y^3,z^2]/(x^6z^2)$ with generators each in degree $d=6$.  Thus the $R$-module $H$ naturally decomposes into 6 submodules and through some computations using the Hilbert series above we obtain the 6 Hilbert polynomials
\begin{align*}
g_0(t)&=4t+1 \qquad &g_1(t)&=4t+1 \qquad &g_2(t)&=4t+2\\
g_3(t)&=4t+3 \qquad &g_4(t)&=4t+3 \qquad &g_5(t)&=4t+4
\end{align*} 
Thus $\cx(\mathbb F_2,\mathbb F_2)=2$ and by Proposition \ref{mult_coeffs} it follows that $e^2(\mathbb F_2,\mathbb F_2)=0$.
\end{example}

\begin{example} Let $k$ be a field, $q\in k$, and $a>0$ an integer. Consider the 
quantum complete intersection
\[
A=A^a_{c,q}=k\langle x_1,\dots,x_c\rangle/(x_i^a,\{x_ix_j-qx_jx_i\}_{i<j})
\]
This is a finite dimensional algebra over $k$.  Let $\sfT=\sfD^b(A)$ be the bounded derived category of $A$.  This is an $\HH^*(A)$-linear triangulated category, where $\HH^*(A)$ is the Hochschild cohomology ring of $A$. As in \cite{BJT24} we may take as $R$ a polynomial subalgebra of $\HH^*(A)$
generated in degree 2. Regarding $A$-modules $X$ and $Y$ as complexes concentrated in degree $0$, one has
$\Hom_{\sfD^b(A)}(X,\Sigma^nY)=\Ext_A^n(X,Y)$.  It follows from \cite[Theorem 5.3]{BO2008} that the Hilbert series of $\Ext^{\ge 0}_A(k,k)$ is given by the rational function 
\[
\frac{1}{(1-t)^c}=\sum_{n\ge 0}{n+c-1 \choose n}t^n
\]
From this one computes the Hilbert polynomials of $\Ext^{\ge 0}_A(k,k)$ 
\[
g_0(t)=\frac{2^{c-1}}{(c-1)!}t^{c-1}+h_0(t) \quad\text{and}\quad g_1(t)=\frac{2^{c-1}}{(c-1)!}t^{c-1}+h_1(t)
\]
where $h_0(t)$ and $h_1(t)$ are polynomials in $\QQ[t]$ of degrees at most $c-2$. We have $\cx(k,k)=c$ and 
$e^c(k,k)=0$.
\end{example}

The next result and example considers when the multiplicities $e^s$ and $e_s$ coincide.
\begin{proposition}\label{symmetry}
Assume that all pairs of objects $(X,Y)$ in $\sfT$ are such that $\Hom_{\sfT}^{*}(X,Y)$ belongs to $\noethfl(R)\cap \artfl(R)$, and that 
$$\lambda^i(X,Y)=0\text{ for all $i\gg0$}\quad \text{if and only if}\quad \lambda^i(X,Y)=0
\text{ for all $i\ll0$}.$$ 
Assume also that $R^0$ is local with infinite residue field. Then for a pair of objects $(X,Y)$ in $\sfT$,
$$\cx(X,Y)=\cx^-(X,Y)\quad \text{ and } \quad e^s(X,Y)=e_s(X,Y)$$ 
for $s\ge \cx(X,Y)$. 
\end{proposition}
\begin{proof}
Let $s=\cx(X,Y)$ and $t=\cx^-(X,Y)$. We argue the case $t\le s$, the other being similar, and induct on $s\ge 0$. For $s=0$, equality $t=s$ holds by assumption. Suppose $s>0$. By Theorem \ref{zee}, there exists an element $z\in R^d$ such that the induced map $\Hom_{\sfT}(X,\Sigma^nY)\to \Hom_{\sfT}(X,\Sigma^{n+d}Y)$ is injective for $n\gg0$ and surjective for $n\ll0$. To see this, let $D(-)=\Hom_{R^0}(-,E(k))$, where $k$ is the residue field of $R^0$ and $E(k)$ its injective envelope. As before, one then has $\Hom_{\sfT}^*(X,Y)$ is in $\artfl(R)$ if and only if $D(\Hom_{\sfT}^*(X,Y))$ is in $\noethfl(R)$; see Kirby \cite{Kir73} and \cite[Theorem 1.4]{BJT25}. One can thus apply Theorem \ref{zee} to $\Hom_{\sfT}^*(X,Y)\oplus D(\Hom_{\sfT}^*(X,Y))$ to get the desired element $z\in R^d$. Moreover, by the proofs of Theorem \ref{zee} and \ref{zee-} one has $\cx(X,Y\slashes z)=s-1$ and $\cx^-(X,Y\slashes z)=t-1$. By induction, we have $s-1=t-1$ and thus $s=t$. 

We now consider the statement about multiplicity. Induct on $s\ge 0$. As observed previously, the base case $e^0(X,Y)=e_0(X,Y)$ holds. For $s>0$, one has
$$e^s(X,Y)=e^{s-1}(X,Y\slashes z)=e_{s-1}(X,Y\slashes z)=e_s(X,Y)$$
using induction in the middle.
\end{proof}

\begin{example}
Let $A$ be a local complete intersection ring and consider the triangulated category $\sfT=\Ktac(A)$ of totally acyclic complexes of projective $A$-modules; it is $R$-linear, where $R=A[\chi_1,...,\chi_c]$ and each $\chi_i$ has degree $d=2$.  For complexes $X$ and $Y$ in $\Ktac(A)$, one has
\begin{align*}
\Hom_{\Ktac(A)}(X,\Sigma^nY)
&=\widehat{\Ext}_A^n(M,N)
\end{align*}
where $M$ and $N$ are the zeroeth image modules in $X$ and $Y$, respectively.  Now by \cite[Theorem 4.7]{AB00}, $\widehat{\Ext}_A^n(M,N)=0$ for $n\gg0$ if and only if $\widehat{\Ext}_A^n(M,N)=0$ for $n\ll0$. Thus by Proposition \ref{symmetry}, we have $e^s(X,Y)=e_s(X,Y)$ in this category.
\end{example}

\begin{remark}
Our multiplicity is an extension of the Herbrand difference considered by Buchweitz \cite[Definition 10.3.1]{Buc86}. Indeed, for a local hypersurface ring $A$ and finitely generated $A$-modules $M$ and $N$ satisfying $\ell\Ext_A^n(M,N)<\infty$ for $n\ge 0$, the Herbrand difference is defined as
$$\ell\Ext_A^{2n}(M,N)-\ell\Ext_A^{2n-1}(M,N)\text{ for $n\gg0$}.$$
This is precisely the multiplicity $e^1(M,N)$ as defined in Definition \ref{Herbrand_dfn}, in the case where $\sfT=\sfD(A)$ is an $R$-linear triangulated category, for $R=A[\chi]$ with $\chi$ of degree $d=2$, and in which case $\Hom_{\sfT}(M,\Sigma^nN)=\Ext_A^n(M,N)$.
\end{remark}

We next show how our multiplicity recovers other invariants, namely, Serre's intersection multiplicity for regular local rings and Hochster's theta invariant. Recall from \cite[V.3]{Ser00} that for a regular local ring $Q$ and finitely generated $Q$-modules $M$ and $N$ satisfying $\ell(M\otimes_QN)<\infty$, Serre's intersection multiplicity of $M$ and $N$ is $\sum_{n=0}^{\dim Q}(-1)^n\ell\Tor_n^Q(M,N)$.

\begin{proposition}\label{same_as_serre}
Let $Q$ be a regular local ring and let $M$ and $N$ be finitely generated $Q$-modules such that $\ell(M\otimes_Q N)<\infty$. Then for $\sfT=\sfDb(Q)$ one has
$$e^0(Q,M\otimes_Q^LN)=\sum_{n=0}^{\dim Q}(-1)^n\ell\Tor_n^Q(M,N).$$
\end{proposition}

\begin{proof}
The triangulated category $\sfT=\sfDb(Q)$ is $Q$-linear, where $Q$ is thought of as a graded ring concentrated in degree $0$. Let $M$ and $N$ be finitely generated $Q$-modules such that $\ell(M\otimes_Q N)<\infty$. Thus also $\ell\Tor_n^Q(M,N)<\infty$ for all $n\ge0$.  In this case, the following hold:
\begin{align*}
\Hom_{\sfDb(Q)}(Q,\Sigma^n(M\otimes_Q^L N))&\cong \Hom_{\sfKb(Q)}(Q,\Sigma^n(M\otimes_Q^L N))\\
&\cong \H^0\Hom_Q(Q,\Sigma^n(M\otimes_Q^L N))\\
&\cong \H^n\Hom_Q(Q,M\otimes_Q^L N)\\
&\cong \H^n(M\otimes_Q^L N)\\
&\cong\Tor^Q_{-n}(M,N)
\end{align*}
Thus $\lambda^n(Q,M\otimes_Q^LN)=\ell \Tor_{-n}^Q(M,N)$. Moreover, we have $\cx(Q,M\otimes_Q^LN)=0$ and $\lambda^n(Q,M\otimes_Q^LN)=0$ for $n\ll0$ and $\Hom_{\sfD}^*(Q,M\otimes_Q^LN)$ is a Noetherian $Q$-module. Thus $e^0(Q,M\otimes_Q^LN)$ is defined, and the result holds.
\end{proof}
More generally, suppose $\sfT$ is a tensor triangulated category with unit object $\mathbb{1}$. Then it is reasonable to consider, for a pair of objects $X$ and $Y$, the invariant
$$e^s(\mathbb{1},X\otimes Y)$$
as a type of intersection multiplicity of $X$ and $Y$. 

Let $A=Q/(f)$ be a local hypersurface ring; that is, $Q$ is a regular local ring and $f$ is a regular element in the square of the maximal ideal of $Q$. For finitely generated $A$-modules $M$ and $N$ with $\ell(M\otimes_AN)<\infty$, recall that Hochster's theta invariant from \cite{Hoc81} is defined as
$$\theta^A(M,N)=\ell\Tor_{2n}^A(M,N)-\ell\Tor_{2n-1}^A(M,N)\text{ for $n\gg0$.}$$

\begin{proposition}\label{same_as_hochster}
Let $A=Q/(f)$ be a local hypersurface ring, and let $M$ and $N$ be finitely generated $A$-modules with $\ell(M\otimes_AN)<\infty$. Then for $\sfT=\sfD(A)$ one has
$$e_1(A,M\otimes_A^LN)=\theta^A(M,N).$$
\end{proposition}
\begin{proof}
The triangulated category $\sfT=\sfD(A)$ is $R$-linear, where $R=A[\chi]$ and $\chi$ is in degree $2$. Let $M$ and $N$ be finitely generated $A$-modules such that $\ell(M\otimes_A N)<\infty$. As before, this implies $\ell\Tor_n^A(M,N)<\infty$ for all $n\ge 0$. Again, the same as in Proposition \ref{same_as_serre}, we have: 
\begin{align*}
\Hom_{\sfD(A)}(A,\Sigma^n(M\otimes_A^LN))
&\cong \H^n(M\otimes_A^L N)
= \Tor_{-n}^A(M,N)
\end{align*}
In particular, $\lambda^n(A,M\otimes_A^LN)=0$ for $n\gg0$ and $\Hom_{\sfD(A)}^{\le 0}(Q,M\otimes_A^LN)$ is an Artinian $R$-module. If $\cx^-(A,M\otimes_A^LN)=0$ (which is the same as saying $\Tor_n^A(M,N)=0$ for $n\gg0$) then both sides vanish. Suppose that $\cx^-(A,M\otimes_A^LN)=1$. In this case, we have by Proposition \ref{mult_coeffs-} that $e_1(A,M\otimes_A^LN)=a_0-a_1$, where for $i=1,2$, the constants $a_i$ are the eventual (constant) negative Hilbert polynomials $g_i^-(n)=\lambda^{2n+i}(A,M\otimes_A^LN)=\ell\Tor_{-(2n+i)}^A(M,N)$. But $a_0-a_1$ is precisely $\theta^A(M,N)$.
\end{proof}

We end this section with an application to vanishing of cohomology.
\begin{theorem}\label{vanishing_app}
Let $X$ and $Y$ be objects in $\sfT$ with $\Hom_{\sfT}^*(X,Y)$ in $\noethfl(R)$.  Let $s=\cx(X,Y)$ and assume that $e^s(X,Y)=0$. There exists an integer $m_0$ with the following property: If $\lambda^n(X,Y)=0$ for $d/2$ consecutive even (or odd) values of $n\ge m_0$, then $\lambda^n(X,Y)=0$ for all $n\ge m_0$.
\end{theorem}
\begin{proof}
By \cite[Theorem 2.3]{BJT24}, there exists an integer $N$ such that if $\lambda^n(X,Y)=0$ for $d$ consecutive values of $n$, then $\lambda^n(X,Y)=0$ for all $n\ge N$. Because $e^s(X,Y)=0$, there must be an integer $m_0\ge N$ such that for any $n\ge m_0$
$$\sum_{i=0}^{d/2-1}\lambda^{n+2i}(X,Y)=\sum_{i=0}^{d/2-1}\lambda^{n+2i+1}(X,Y)$$
and so the result follows.
\end{proof}

\begin{corollary}
Let $A$ be a local complete intersection ring, and let $M$ and $N$ be finitely generated $A$-modules such that $\Ext_A^n(M,N)$ has finite length for $n\ge 0$. If $\cx(M,N)=r$, $e^r(M,N)=0$, and $\Ext_A^n(M,N)=0$ for infinitely many even $n\gg0$, or infinitely many odd $n\gg0$, then $\Ext_A^n(M,N)=0$ for all $n> \dim A-\depth M$.
\end{corollary}
\begin{proof}
If $A$ has codimension $c$, then apply Theorem \ref{vanishing_app} to the $R$-linear triangulated category $\sfT=\sfD(A)$, where $R=A[\chi_1,...,\chi_c]$, to obtain that $\Ext_A^n(M,N)=0$ for $n\gg0$. From \cite[Theorem 4.2]{AY98}, $\Ext_A^n(M,N)=0$ for $n>\dim A- \depth M$.
\end{proof}

\section{A limit approach}\label{section_lim}
Let $X$ and $Y$ be objects in $\sfT$. Assume in this section that $\Hom_{\sfT}^{*}(X,Y)$ belongs to $\noethfl(R)$.  As in Section \ref{section_mult}, we may assume that $R$ is generated over $R^0$ by elements of even degree $d$, and we let $g_0(t),...,g_{d-1}(t)$ be the Hilbert polynomials of $(X,Y)$ so that $g_i(n)=\lambda^{dn+i}(X,Y)$ for all $n\gg0$. Letting $s\ge \cx(X,Y)$, for each $i=0,...,d-1$ one may write $g_i(t)=a_it^{s-1}+\text{(lower degree terms)}$. As the degree of $g_i(t)$ is at most $s-1$, note that $a_i$ might be zero for some $i$.
The goal of this section is to prove that the multiplicity $e^s(X,Y)$ can be expressed as a limit, analogous to the classic Hilbert-Samuel multiplicity \cite[Formula 14.1]{Mat89}, and similar to the ones considered in \cite{Dao07,CD11}. 

We start with a fact that follows from the classic Faulhaber's formula:
\begin{lemma}\label{Faulhaber}
Let $g(t)=c_rt^r+c_{r-1}t^{r-1}+\cdots+c_0 \in \QQ[t]$ be a degree $r$ polynomial, with $c_r\not=0$. Fix an integer $N\ge 0$. The function $\sigma(n)=\sum_{i=N}^ng(i)$ is a polynomial in $n$ of degree $r+1$:
$$\sigma(n)=\sum_{i=N}^ng(i)=\frac{c_r}{r+1}n^{r+1}+h(n)$$
where $h(t)\in \QQ[t]$ is a polynomial of degree at most $r$.
\end{lemma}
\begin{proof}
For a fixed $N\ge 0$, note first that $\sum_{i=0}^{N-1}g(i)$ is a constant, so we may simply consider the sum $\sigma(n)=\sum_{i=0}^ng(i)$. Next, by linearity, it is enough to consider polynomials of the form $g(t)=t^r$. In this case, the result follows from Faulhaber's formula, see for example \cite[p. 106]{CG96}:
$$1^r+2^r+\cdots +n^r=\frac{1}{r+1}n^{r+1}+h(n)$$
where $h(t)\in \QQ[t]$ is a polynomial of degree at most $r$.
\end{proof}

\begin{theorem}\label{limit_thm}
Let $X$ and $Y$ be objects in $\sfT$ with $\Hom_{\sfT}^{*}(X,Y)$ in $\noethfl(R)$. Let $s\ge \cx(X,Y)\ge 1$ be an integer.  Then
$$e^s(X,Y)=s!d^{2s-1}\lim_{n\to \infty} \frac{\sum_{j=0}^n(-1)^j\lambda^j(X,Y)}{n^s}$$
\end{theorem}
\begin{proof}
Select $N\ge 0$ so that for $i=0,...,d-1$ one has $\lambda^{dj+i}(X,Y)=g_i(j)$ for all $j\ge N$. For each $i=0,...,d-1$, write $g_i(t)=a_it^{s-1}+\text{(lower degree terms)}$, a polynomial in $\QQ[t]$.

First consider the sum $\sum_{j=0}^n(-1)^j\lambda^j(X,Y)$. Writing $n=md+a$ for integers $m\ge 0$ and $0\le a \le d-1$, one has $\lfloor \frac{n-i}{d}\rfloor=\frac{n-a}{d}=m$ if $0\le i\le a$ and $\lfloor \frac{n-i}{d}\rfloor=\frac{n-a-d}{d}=m-1$ if $a<i\le d-1$. Use this to observe the following:
\begin{align}\label{splitup}
\sum_{j=dN}^n(-1)^j\lambda^j(X,Y)&=\sum_{i=0}^{d-1}(-1)^i\sum_{j=N}^{\lfloor \frac{n-i}{d}\rfloor}\lambda^{dj+i}(X,Y) =\sum_{i=0}^{d-1}(-1)^i\sum_{j=N}^{\lfloor \frac{n-i}{d}\rfloor}g_{i}(j)
\end{align}

Now taking limits, as $s\ge 1$ we may disregard the constant $\sum_{j=0}^{dN-1}(-1)^j\lambda^j(X,Y)$ in the original sum. From \eqref{splitup} and Lemma \ref{Faulhaber}, one has:
\begin{align*}
\lim_{n\to \infty} \frac{\sum_{j=0}^n(-1)^j\lambda^j(X,Y)}{n^s}&=\lim_{n\to \infty} \frac{1}{n^s}\sum_{j=dN}^n(-1)^j\lambda^j(X,Y)\\
&=\lim_{n\to \infty} \frac{1}{n^s} \sum_{i=0}^{d-1}(-1)^i\sum_{j=N}^{\lfloor \frac{n-i}{d}\rfloor}g_i(j)\nonumber\\
&=\lim_{n\to \infty} \frac{1}{n^s} \sum_{i=0}^{d-1}(-1)^i\frac{a_i}{s}\left(\lfloor\textstyle{\frac{n-i}{d}\rfloor}^s+h_i\left(\lfloor\textstyle{\frac{n-i}{d}\rfloor}\right)\right)\nonumber
\end{align*}
where each $h_i$ is a polynomial of degree at most $s-1$. Using again the description given above for $\lfloor\frac{n-i}{d}\rfloor$ along with the binomial theorem, one obtains the equality  $\lfloor \frac{n-i}{d}\rfloor^s+h_i(\lfloor\frac{n-i}{d}\rfloor)=\frac{n^s}{d^s}+\text{(lower degree terms)}$, and thus the limit simplifies to:
\begin{align*}
\lim_{n\to \infty} \frac{\sum_{j=0}^n(-1)^j\lambda^j(X,Y)}{n^s}&= \sum_{i=0}^{d-1}(-1)^i\frac{a_i}{sd^s}
\end{align*}

By Proposition \ref{mult_coeffs}, 
$$e^s(X,Y)=\sum_{i=0}^{d-1}(-1)^i a_i(s-1)! d^{s-1}$$
Thus one has
\begin{align*}
\lim_{n\to \infty} \frac{\sum_{j=0}^n(-1)^j\lambda^j(X,Y)}{n^s}&=\frac{e^s(X,Y)}{s!d^{2s-1}}
\end{align*}
which yields the desired result.
\end{proof}

\appendix
\section{Difference operators}
In this appendix we prove an elementary result for difference operators $\Delta^1$ of index $d$. In the case $d=1$, this is standard, see for example \cite[1.9]{Sta12}, but for $d>1$ (needed in this paper) we were unable to find a suitable reference.

\begin{chunk}\label{proofs_of_closed_forms}
First, we prove the closed expression of $\Delta^s$ from \eqref{Deltas+} in Section \ref{section_herbrand}.
It holds trivially when $s=0$. Now assume it holds for $s-1\ge 0$.  Then we have
\begin{align*}
&\Delta^sf(n)=\Delta^1(\Delta^{s-1}f)(n)=\Delta^{s-1}f(n+d)-\Delta^{s-1}f(n)\\
&=\sum_{i=0}^{s-1}(-1)^i\binom{s-1}{i}f(n+d+(s-1-i)d)-
\sum_{i=0}^{s-1}(-1)^i\binom{s-1}{i}f(n+(s-1-i)d)\\
&=\sum_{i=0}^{s-1}(-1)^i\binom{s-1}{i}f(n+(s-i)d)+
\sum_{i=1}^s(-1)^i\binom{s-1}{i-1}f(n+(s-i)d)\\
&=\sum_{i=0}^s(-1)^i\binom{s-1}{i}f(n+(s-i)d)+
\sum_{i=0}^s(-1)^i\binom{s-1}{i-1}f(n+(s-i)d)\\
&=\sum_{i=0}^s(-1)^i\left(\binom{s-1}{i}+\binom{s-1}{i-1}\right)f(n+(s-i)d)\\
&=\sum_{i=0}^s(-1)^i\binom{s}{i}f(n+(s-i)d)
\end{align*}
\end{chunk}

\begin{lemma} \label{10} For $s\ge 1$ and $0\le n<s$ we have $\sum_{i=0}^s(-1)^i{s\choose i} i^n=0$.
\end{lemma}

\begin{proof}
We induct on $s$. For $s=1$ we have $n=0$, and the expression is ${1\choose 0}-{1\choose 1}=0$.
Assume the result holds for $s-1$ and any $0\le n<s-1$.  For $0\le n<s$ we have
\begin{align*}
\sum_{i=0}^s(-1)^i{s\choose i} i^n=&\sum_{i=0}^s(-1)^i\left({s-1\choose i-1}+{s-1\choose i}\right)i^n\\
=&\sum_{i=0}^s(-1)^i{s-1\choose i-1}i^n+\sum_{i=0}^s(-1)^i{s-1\choose i}i^n\\
=&\sum_{i=0}^{s-1}(-1)^{i+1}{s-1\choose i}(i+1)^n+\sum_{i=0}^{s-1}(-1)^i{s-1\choose i}i^n\\
=&\sum_{i=0}^{s-1}(-1)^i{s-1\choose i}\left(i^n-(i+1)^n\right)\\
=&\sum_{i=0}^{s-1}(-1)^i{s-1\choose i}\left(i^n-\sum_{j=0}^n{n\choose j}i^{n-j}\right)\\
=&\sum_{i=0}^{s-1}(-1)^i{s-1\choose i}\left(-\sum_{j=1}^n{n\choose j}i^{n-j}\right)\\
=&-\sum_{j=1}^n{n\choose j}\left(\sum_{i=0}^{s-1}(-1)^i{s-1\choose i}i^{n-j}\right)=0
\end{align*}
by induction, since $n-j<s-1$ for $1\le j\le n$. 
\end{proof}

\begin{corollary}\label{20} For integers $s\ge 1$, $0\le n<s$, and any integers $m,d$ we have
$\sum_{i=0}^s(-1)^i{s\choose i} (m+di)^n=0$.
\end{corollary}

\begin{proof} We have
\begin{align*}
\sum_{i=0}^s(-1)^i{s\choose i}(m+di)^n=&\sum_{i=0}^s(-1)^i{s\choose i}\left(\sum_{j=0}^n{n\choose j}m^{n-j}(di)^j\right)\\
=&\sum_{j=0}^n{n\choose j}m^{n-j}d^j\left(\sum_{i=0}^s(-1)^i{s\choose i}i^j\right)=0
\end{align*}
since each term in parentheses is zero from Lemma \ref{10}.
\end{proof}

Recall from \ref{difference_operators} that $\Delta^s$ is the difference operator of index $d$.
\begin{theorem}\label{diffpower} 
For a polynomial $g(t)=at^r+\emph{(lower degree terms)}$ in $\QQ[t]$ of degree $r\ge0$, and an integer $s\ge r$, one has that $\Delta^sg(t)=\begin{cases} 
  0  & \text{ if $s>r$} \\
  as!d^s & \text{ if $s=r$}
\end{cases}$
\end{theorem}

\begin{proof} 
By linearity of $\Delta^s$, it suffices to consider the case $g(t)=t^r$.
Assume first that $s>r$. From \ref{difference_operators}\eqref{Deltas+} we have 
\begin{align*}
\Delta^st^r=&\sum_{i=0}^s(-1)^i\binom{s}{i}(t+(s-i)d)^r\\
=&\sum_{i=0}^s(-1)^i\binom{s}{i}\left(\sum_{j=0}^r{r\choose j}t^{r-j}(s-i)^jd^j\right)\\
=&\sum_{j=0}^r{r\choose j}t^{r-j}d^j\left(\sum_{i=0}^s(-1)^i\binom{s}{i}(s-i)^j\right)=0
\end{align*}
since each term in parentheses is zero by Corollary \ref{20}.

Now assume that $s=r$.  We prove the result by induction on $s$.  For $s=0$ the result is trivial, and for $s=1$ we have 
$\Delta^1t=(t+d)-t=d$, which is what we wanted.  Now assume $s>1$ and the formula holds for $s-1$.  From \eqref{Deltas+} and Pascal's identity one has
\begin{align*}
\Delta^st^s=&\sum_{i=0}^s(-1)^i\binom{s}{i}(t+(s-i)d)^s\\
=&\sum_{i=0}^s(-1)^i\left({s-1\choose i-1}+{s-1\choose i}\right)(t+(s-i)d)^s\\
=&\sum_{i=0}^s(-1)^i{s-1\choose i-1}(t+(s-i)d)^s+\sum_{i=0}^s(-1)^i{s-1\choose i}(t+(s-i)d)^s\\
=&\sum_{i=0}^{s-1}(-1)^{i+1}{s-1\choose i}(t+(s-1-i)d)^s+\sum_{i=0}^{s-1}(-1)^i{s-1\choose i}(t+(s-i)d)^s\\
=&\sum_{i=0}^{s-1}(-1)^i{s-1\choose i}\left((t+(s-i)d)^s-(t+(s-1-i)d)^s\right)\\
=&\sum_{i=0}^{s-1}(-1)^i{s-1\choose i}\left((t+(s-1-i)d+d)^s-(t+(s-1-i)d)^s\right)
\end{align*}
From this, use the binomial theorem to subsequently obtain
\begin{align*}
\Delta^st^s=&\sum_{i=0}^{s-1}(-1)^i{s-1\choose i}\left(\sum_{j=0}^s{s\choose j}(t+(s-1-i)d)^{s-j}d^j-(t+(s-1-i)d)^s\right)\\
=&\sum_{i=0}^{s-1}(-1)^i{s-1\choose i}\left(\sum_{j=1}^s{s\choose j}(t+(s-1-i)d)^{s-j}d^j\right)\\
=&\sum_{j=1}^s{s\choose j}d^j\left(\sum_{i=0}^{s-1}(-1)^i{s-1\choose i}(t+(s-1-i)d)^{s-j}\right)\\
=&sd\left(\sum_{i=0}^{s-1}(-1)^i{s-1\choose i}(t+(s-1-i)d)^{s-1}\right)\\
=&sd\left((s-1)!d^{s-1}\right)=s!d^s
\end{align*}
where the penultimate line is by Corollary \ref{20}, and the last line is by induction.
\end{proof}


\end{document}